\documentclass[preprint,10pt]{elsarticle}

\usepackage{amsmath,amssymb}
\usepackage{amsfonts}
\usepackage{mathrsfs}
\usepackage{color}
\usepackage{graphicx}
\usepackage{epsfig} 

\usepackage{array,booktabs} 
\usepackage{amsmath,amssymb} 
\usepackage{stfloats} 
\usepackage{psfrag} 
\usepackage{color}
\usepackage{natbib}

\definecolor{box_color}{rgb}{.8,.8,.8}

\usepackage[utf8]{inputenc}
\usepackage[T1]{fontenc}
\usepackage{lmodern}
\usepackage{epstopdf}

\newtheorem{lemma}{Lemma}
\newtheorem{proposition}{Proposition}
\newtheorem{corollary}{Corollary}
\newtheorem{fact}{Fact}
\newtheorem{definition}{Definition}
\newtheorem{remark}{Remark}

\newtheorem{assumption}{Assumption}
\newtheorem{ass}{C.}

\newenvironment{proof}[1][Proof]{\begin{trivlist}
\item[\hskip \labelsep {\bfseries #1}]}{\end{trivlist}}

\def\begcen{\begin{center}}
\def\endcen{\end{center}}

\newcommand{\bfp}{\mbox{$p$}}

\def\frp{\mathfrak{p}}

\newcommand{\RE}{\mathbb {R}}    

\newcommand{\col}{ \mbox{col} }

\def\calm{{\cal M}}

\def\calV{{\cal V}}
\def\calG{{\cal G}}
\def\calH{{\cal H}}

\def\hal{{1 \over 2}}

\def\L2{{\cal L}_2}
\def\L2e{{\cal L}_{2e}}

\def\rea{\mathbb{R}}

\def\diag{\mbox{diag}}

\def\half{\frac{1}{2}}

\def\begequarr{\begin{eqnarray}}
\def\endequarr{\end{eqnarray}}
\def\begequarrs{\begin{eqnarray*}}
\def\endequarrs{\end{eqnarray*}}
\def\begarr{\begin{array}}
\def\endarr{\end{array}}
\def\begequ{\begin{equation}}
\def\endequ{\end{equation}}
\def\lab{\label}
\def\begdes{\begin{description}}
\def\enddes{\end{description}}
\def\begenu{\begin{enumerate}}
\def\begite{\begin{itemize}}
\def\endite{\end{itemize}}
\def\endenu{\end{enumerate}}

\def\lef[{\left[\begin{array}}
\def\rig]{\end{array}\right]}
\def\qed{\hfill$\Box \Box \Box$}
\def\begcen{\begin{center}}
\def\endcen{\end{center}}
\def\begrem{\begin{remark}\rm}
\def\endrem{\end{remark}}
\def\begassum{\begin{assumption}}
\def\endassum{\end{assumption}}
\def\begassums{\begin{assumption*}}
\def\endassums{\end{assumption*}}
\def\begassu{\begin{ass}}
\def\endassu{\end{ass}}
\def\beglem{\begin{lemma}}
\def\endlem{\end{lemma}}
\def\begcor{\begin{corollary}}
\def\endcor{\end{corollary}}
\def\begfac{\begin{fact}}
\def\endfac{\end{fact}}

\def\bfp{\mathbf{p}}

\journal{Systems \& Control Letters}

\begin{document}

\begin{frontmatter}



\title{Simultaneous Interconnection and Damping Assignment Passivity--based Control of Mechanical Systems Using Generalized Forces}


\author[label1]{A. Donaire\corref{cor1}}
\ead{Alejandro.Donaire@newcastle.edu.au}
\cortext[cor1]{Corresponding author at PRISMA Lab, Dipartimento di Ingegneria Elettrica e Tecnologie dell'Informazione, Universit\`a degli Studi di Napoli, via Claudio 21, 80125, Naples, Italy, Tel.: +39 817 683513.}
\author[label2]{R. Ortega}
\ead{ortega@lss.supelec.fr}
\author[label3]{J.G. Romero}
\ead{Jose.Romero-Velazquez@lirmm.fr}
\address[label1]{PRISMA Lab, University of Naples Federico II, Italy, and School of Engineering, The University of Newcastle, Australia.}
\address[label2]{Laboratoire des Signaux et Syst\`emes, CNRS--SUPELEC, Plateau du Moulon, 91192 Gif--sur--Yvette, France.}
\address[label3]{Laboratoire d'Informatique, de Robotique et de Micro\'electronique de Montpellier,  Montpellier, France.}

\begin{abstract}
To extend the realm of application of the well known controller design technique of  interconnection and damping assignment passivity--based control (IDA--PBC) of mechanical systems two modifications to the standard method are presented in this article. First, similarly to \cite{BATetal}, it is proposed to avoid the splitting of the control action into energy--shaping and damping injection terms, but instead to carry them out {\em simultaneously}. Second, motivated by \cite{CHA}, we propose to consider the inclusion of  {\em generalised forces}, going beyond the gyroscopic ones used in standard IDA--PBC. It is shown that several new controllers for mechanical systems designed invoking other (less systematic procedures) that do not satisfy the conditions of standard IDA--PBC, actually belong to this new class of SIDA--PBC.
\end{abstract}


\end{frontmatter}

\section{Introduction}
\lab{sec1}
Stabilization of underactuated mechanical systems shaping their potential energy function, and preserving the systems structure, is a simple, robust and highly successful technique first introduced in \cite{TAKARI}. To enlarge its realm of application it has been proposed to modify the kinetic energy of the system as well. This idea of total energy shaping was first introduced in \cite{AILORT} with the two main approaches being now: the method of controlled Lagrangians \cite{BLOLEOMAR} and Interconnection and Damping Assignment Passivity-Based Control (IDA--PBC) \cite{ORTtac}, see also the closely related work \cite{FUJSUG}. In both cases stabilization (of a desired equilibrium) is achieved identifying the class of systems---Lagrangian for the first method and Hamiltonian for IDA--PBC---that can possibly be obtained via feedback. The conditions under which such a feedback law exists are identified by the so--called \emph{matching equations}, which are a set of quasi-linear partial differential equations (PDEs), that are naturally split into kinetic energy (KE--PDE) and potential energy (PE--PDE).

Although a lot of research effort has been devoted to the solution of the matching equations---see \cite{CRAORTPIL,Donaire2015} for a recent survey of the existing results---this task remains the main stumbling block for the application of these methods. The solution of the KE--PDE is simplified by the inclusion of {\em gyroscopic forces} in the target dynamics, which translates into the presence of a free skew-symmetric matrix in the matching equation that reduces the number of PDEs to be solved. Due to its Hamiltonian formulation, this term is intrinsic in IDA--PBC, and was added to the original controlled Lagrangian method of \cite{BLOLEOMAR,BLOetal}---for the first time in \cite{BLAORTVAN}---and adopted later in \cite{CHAetal}. In \cite{BLAORTVAN} it is shown that the PDEs of the (extended) controlled Lagrangian method and IDA--PBC are the same, see also \cite{CHAetal}.

Recently, in \cite{CHA} it has been proposed to consider a more general form for these forces, relaxing the skew-symmetry condition. It is claimed in \cite{CHA} that the inclusion of these forces {\em reduces the number} of KE--PDEs, but as shown in \cite{CRAORTPIL} this claim turned out to be {\em wrong}. One of the objectives of this paper is to show that, even though the number of PDEs is not reduced, the inclusion of generalised forces effectively extends the realm of application of IDA--PBC. A second modification to IDA--PBC proposed in the paper is to simultaneously carry out the energy shaping and damping injection steps---instead of doing them as separate steps. This modification has been previously reported in  \cite{BATetal}, where it is shown that the partition into two steps of the design procedure induces some loss of generality. In particular, it is shown that (two--step) IDA--PBC is not applicable for the induction motor, while SIDA--PBC does apply.

In the paper we also show that several recent controller designs that do not fit in the standard IDA--PBC paradigm, actually belong to this new  class of SIDA--PBC with generalised forces. In this way, it is shown that these controllers, that were derived invoking less systematic procedures,  are obtained following the well--established SIDA--PBC methodology. 

The remaining of the paper is organized as follows. Section \ref{sec2} briefly recalls the IDA--PBC methodology.  Section \ref{sec3} contains the main result, which is the definition of SIDA--PBC with generalised forces. Two recently reported controller design techniques are shown to belong to this class in Section \ref{sec4}. The paper is wrapped--up with concluding remarks in  Section \ref{sec5}.\\

\noindent {\bf Notation.} $I_n$ is the $n \times n$ identity matrix and $0_{n \times s}$ is an
$n \times s$ matrix of zeros, $0_n$ is an $n$--dimensional column vector of zeros. Given $a_i \in \rea,\; i \in \bar n := \{1,\dots,n\}$, we denote with $\col(a_i)$ the $n$--dimensional column vector with
elements $a_i$. For any matrix $A \in \rea^{n \times n}$, $(A)_{i} \in \rea^n$ denotes the $i$--th column, $(A)^{i}$ the $i$--th row and $(A)_{ij}$ the $ij$--th element. $e_i\in \rea^n,\; i \in
\bar n$, is the Euclidean basis vectors. For $x \in \rea^n$, $S \in \rea^{n \times n}$, $S=S^\top
>0$, we denote the Euclidean norm $|x|^2:=x^\top x$, and the weighted--norm $\|x\|^2_S:=x^\top S x$. Given a function $f:  \rea^n \to \rea$ we define the differential operators
$$
\nabla_x f:=\left(\frac{\displaystyle \partial f }{\displaystyle \partial x}\right)^\top,\;\nabla_{x_i} f:=\left(\frac{\displaystyle
\partial f }{\displaystyle \partial x_i}\right)^\top,
$$
where $x_i \in \rea^p$ is an element of the vector $x$. For a mapping $g : \rea^n \to \rea^m$, its Jacobian matrix is defined as
$$\nabla g:=\left [\begin{array}{cc}(\nabla g_1)^\top \\
\vdots\\ (\nabla g_m)^\top \end{array}\right],$$ where $g_i:\rea^n \to \rea$ is the $i$-th element of $g$. When clear from the context the subindex in $\nabla$ will be omitted.  To simplify the expressions, the arguments of all mappings will be omitted, and will be explicitly written only the first time that the mapping is defined.
%
\section{Standard Interconnection and Damping Assignment PBC}
\label{sec2}
%
To make the paper self--contained a brief review of IDA--PBC is presented in this section. IDA--PBC was introduced in \cite{ORTtac} to control underactuated mechanical systems described in port--Hamiltonian (pH) form by
\begin{eqnarray}
\Sigma : \; \left[ \begin{array}{c}
                   \dot q \\
                   \dot p
                   \end{array} \right] = \left[ \begin{array}{cc}
                                               0_{n\times n}      & I_{n} \\
                                              -I_{n}  & 0_{n\times n}
                                              \end{array} \right] \nabla H(q,p)+ \left[ \begin{array}{c}
                                                                        0_{n\times m} \\
                                                                        G(q)
                                                                  \end{array}\right] \, u,
\label{sys}
\end{eqnarray}
\noindent where $q,p \in \rea^{n}$ are the generalized position and momenta, respectively, $u \in \rea^{m}$ is the control, $G \colon \rea^{n} \to \rea^{n \times m}$ with $\mbox{\rm rank}(G)= m < n$, the function $H \colon \rea^{n} \times \rea^{n} \to \rea,$
\begin{equation}
\lab{sys0}
H(q,p) := {1 \over 2} \, p^\top \,  M^{-1}(q) \, p + V(q)
\end{equation}
is the total energy with $M \colon \rea^{n}  \to \rea^{n\times n}$, the positive definite inertia matrix and $V \colon \rea^{n} \to \rea$ the potential energy. The control objective is to generate a state--feedback control  that assigns to the closed-loop the stable equilibrium  $(q,p) = (q^{\star},0)$, $q^{\star} \in \rea^{n}$. This is achieved in IDA--PBC via a two step procedure. The first one, called energy shaping, determines a state--feedback to match the pH target dynamics
\begequ
\label{sysd}
\Sigma_d :\; \lef[{c}
             \dot{q} \\
             \dot{p}
             \rig] = \lef[{cc}
                      0_{n\times n}            &  M^{-1}(q) \, M_d(q)   \\
                     -M_d(q) \, M^{-1}(q)  &  J_2(q,p)
                     \rig] \nabla H_d(q,p)
                     \endequ
with the new total energy function $H_{d} \colon \rea^{n} \times \rea^{n} \to \rea,$
\begin{equation}
\lab{hd}
H_{d}(q,p) := {1 \over 2} \, p^\top \, M_{d}^{-1}(q)\, p + V_{d}(q),
\end{equation}
where $M_{d} \colon \rea^{n}  \to \rea^{n \times n}$ is positive definite,  {\bf $V_{d} \colon \rea^{n} \to \rea$} verifies
\begequ
\lab{vdmin}
q_{\star}   =  \arg \min V_{d}(q),
\endequ
and $J_{2} \colon \rea^{n} \times \rea^{n} \to \rea^{n \times n}$ fulfills the skew--symmetry condition
\begequ
\lab{j2}
J_{2}(q,p)  =  -J_{2}^\top(q,p).
\endequ
In this case, $(q^{\star},0)$ is a stable equilibrium point of (\ref{sysd}) with Lyapunov function $H_{d}.$ Indeed, the time derivative of $H_{d}$ along the
trajectories of \eqref{sysd} takes the form
$$
\dot{H}_{d} =  p^\top \, M_d^{-1} \, J_{2} \, M_{d}^{-1} \, p \equiv 0.
$$
The second step, called -damping injection, is aimed at achieving asymptotic stability. This step is carried out feeding back the natural passive output, that is, adding to the energy shaping control a term of the form $-K_p G^\top M_d^{-1}p$, with $K_P \in \rea^{n \times n}$ positive definite. With this new term we get
$$
\dot{H}_{d} = -\| G^\top M_d^{-1}p\|_{K_P}^2 \leq 0.
$$
Asymptotic stability follows if  the output $G^\top M_d^{-1}p$ is detectable \cite{VANbook}.

To determine the energy--shaping control we equate the right-hand sides of \eqref{sys} and \eqref{sysd} to obtain the so--called matching equations
\begequ
\lab{matequ0}
\nabla_{q} H - G \, u  = M_{d} \, M^{-1} \, \nabla_{q} H_{d} - J_{2} \, M_{d}^{-1} \,p.
\endequ
As shown in \cite{ORTtac} these equations are equivalent to the solution of the ($p$-dependent) KE--PDE
\begequ
G^\perp \left\{  \nabla_q(p^\top  M^{-1} p) - M_{d} \, M^{-1} \, \nabla_q(p^\top   M_{d}^{-1} p)  +  2 \, J_{2} \, M_{d}^{-1} \, p \right \} = 0_{s}, \label{pde1}
\endequ
the ($p$-independent) PE--PDE
\begequ
\lab{pde2}
G^\perp \{ \nabla{V} - M_{d} \, M^{-1} \, \nabla{V_d} \} = 0_{s},
\endequ
and the (univocally defined) control
\begin{equation}
u =  (G^\top \, G)^{-1} \, G^\top \left[ \nabla_{q} H - M_{d} \, M^{-1} \, \nabla_{q}{H_d} + J_2 \, M_d^{-1} \, p \right],
\label{u}
\end{equation}
where $G^{\perp} \colon \rea^{n} \to \rea^{s \times n},\;s:=n-m$ is a full rank left annihilator of $\;G$, {\em i.e.}, $ G^{\perp} G=0_{s\times m}$ and $\mbox{\rm rank}(G^\perp)=s$.

The success of IDA--PBC relies on the possibility of solving the PDEs (\ref{pde1}) and (\ref{pde2}). As shown below, the inclusion of generalised forces affects only the KE--PDE (\ref{pde1}),  therefore in the sequel we concentrate our attention on the KE-PDE \eqref{pde1}. In \cite{CRAORTPIL} a more explicit expression for this equation is obtained as follows. First, note that to be consistent with (\ref{pde1}), whose remaining terms are quadratic in $p,$ the free matrix $J_2$ {\em must be linear} in $p$. Hence,  without loss of generality we can take $J_2$ of the form
\begequ
\label{defJ2}
J_{2}(q,p) = \sum_{i=1}^{n} e_i^\top M_{d}^{-1}p \, U_{i}(q),
\endequ
where $U_{i} \colon \rea^{n} \to \rea^{n \times n}$ verify $U_{i}(q) = -U_{i}^\top(q)$ . To streamline the presentation of the result of  \cite{CRAORTPIL} we denote the columns of $G^\perp$ as
$$
G^{\perp} (q)=: \begin{bmatrix}
                v_{1}^\top(q) \\
                \vdots \\
                v_{s}^\top (q)
                \end{bmatrix},
$$
where $v_k:\rea^n \to \rea^n,\;k \in \bar s:=\{1,\dots,s\}$ is given by $v_{k} := \col( v_{ki}).$ Also, we introduce the mappings
$$
A_{k}: \rea^n \to  \rea^{n \times n}, \; B_{k}: \rea^n \to \rea^{n \times n},\;\Gamma_{k j}: \rea^n \to   \rea,\;W_{k}: \rea^n \to  \rea^{n \times n}.
$$
as
\begequarrs
A_{k} &  := & M_{d} \,\left (\sum_{i=1}^{n} v_{k i} \, \nabla_{q_i} M^{-1} \right ) M_{d},\;k \in \bar s \\
\Gamma_{k j} & := & \sum_{i=1}^{n}  v_{k i} \, (M_{d} \, M^{-1})_{ij},\;k \in \bar s,\; j \in \bar n:=\{1,\dots,n\} \\
B_{k} & := & M_{d} \,\left ( \sum_{i=1}^{n} \,\Gamma_{k i}  \, \nabla_{q_i} M_{d}^{-1} \right ) \, M_{d},\;k \in \bar s \\
W_{k} & := & \begin{bmatrix}
                  v_{k}^\top \, U_{1} \\
                  \vdots          \\
                  v_{k}^\top \, U_{n}
                  \end{bmatrix} + \begin{bmatrix}
                                  v_{k}^\top \, U_{1} \\
                                  \vdots          \\
                                  v_{k}^\top\, U_{n}
                                  \end{bmatrix}^\top,\;k \in \bar s.
\endequarrs

The proof of the lemma below is given in   \cite{CRAORTPIL}.
\begin{lemma}\em
\lab{lem1}
The KE--PDE \eqref{pde1} is equivalent to the PDEs
\begequ
\lab{ke}
B_k(q)-A_k(q)=W_k(q),\;k \in \bar s.
\endequ
\end{lemma}

Note that  the left-hand-side of \eqref{ke} is a function of the unknown matrix $M_d$ (and partial derivatives of its components), while the right-hand-side of \eqref{ke} is {\em independent} of the unknown matrix $M_d$ (and partial derivatives of its components). Hence the number of free elements on the right-hand-side of \eqref{ke} entirely determines the number of KE--PDE's to be solved. It is shown in  \cite{CRAORTPIL} that this number equals
\begequ
\lab{numpde}
 \displaystyle{\frac{1}{6}\,s\,(s+1)\,(s+2)}.
\endequ
Also, contrary to the claim in \cite{CHA},  the {\em explicit formula} \eqref{ke}---given in a different form also in \cite{ACOetal}---shows that there is no {\em ansatz} for the determination of $J_2$ in IDA--PBC.

%
\section{Simultaneous IDA--PBC with Generalized Forces}
\label{sec3}
%
In this paper, motivated by \cite{CHA}, we investigate the possibility to extending the realm of application of IDA--PBC  by considering more general external forces. In \cite{CHA} it is proposed to replace the target dynamics $\Sigma_d$ in \eqref{sysd} by
\begequ
\label{tardyncha}
\Sigma_T :\; \lef[{c}
                  \dot{q} \\
                  \dot{p}
                  \rig] = \lef[{cc}
                              0_{n\times n}             &  M^{-1} (q)\, M_d(q) \\
                           -M_d (q)\, M^{-1}(q)    &  0_{n\times n}
                          \rig]    \nabla H_d (q,p)+ \lef[{c}
                                         0 \\
                                         C(q,p)
                                        \rig],
\endequ
where $C \colon \rea^n \times \rea^n \to \rea^n$ is a mapping to be defined. Notice that, to ensure $H_d$ is a Lyapunov function of the closed--loop---{\em i.e.}, $\dot H_d \leq 0$---the mapping $C$ should satisfy
\begequ
\lab{pmdmin}
p^{{\top} }M_d^{-1} C \leq 0.
\endequ
Since $\Sigma_T$ and $\Sigma_d$ coincide for the particular choice $C = J_2 \, M_d^{-1}p$, it is clear that considering these more general forces enlarges the set of desired closed--loop dynamics.

The matching equation now takes the form
\begequ
\lab{matequ1}
-\hal \, \nabla_q(p^\top \,M^{-1}\,p)- \nabla V + G\, u = - M_d \, M^{-1} \, \left[\hal \, \nabla_q(p^\top  \, M_d^{-1}\, p) + \nabla V_d \right] + C,
\endequ
the KE--PDE \eqref{pde1} becomes
\begequarr
G^\perp \big\{  \nabla_q(p^\top \, M^{-1}\,p) - M_d \, M^{-1}\, \nabla_q(p^\top \, M_d^{-1} \, p) + 2 \, C \big\} = 0_{s}, 
\label{kepde}
\endequarr
while the PE--PDE \eqref{pde2} remains unchanged. Stemming from the equation above we have two important observations regarding $C$ .
\begenu
\item[{\bf O1.}] Since $C(q,0)=0_n$ must be satisfied, $C$ can always be expressed in the form
$$
C(q,p)=\Lambda(q,p)M_d^{-1}(q)p,
$$
for some mapping $\Lambda:\rea^n \times \rea^n \to \rea^{n \times n}$.
\item[{\bf O2.}] $C$ must be quadratic in $p$---this in contrast to the case of $J_2$ that is linear in $p$. For convenience, and without loss of generality, we take it of the form
$$
2 \, C(q,p) = \sum_{i=1}^{n}  \left ( p^\top M_d^{-1}(q)\, Q_{i}(q) \, M_d^{-1}(q)p \right ) \, e_{i}
$$
with $Q_i:\rea^n \to \rea^{n \times n}$ {\em free matrices}. Consequently, we have
\begequ
\lab{lam}
\Lambda(q,p):=\hal \sum_{i=1}^{n} e_i  p^\top M_d^{-1}(q)\, Q_{i}(q).
\endequ
\endenu

Two consequences of the remarks above are, on one hand, that the target dynamics $\Sigma_T$ can be written in the familiar form
\begequ
\lab{sigt}
\Sigma_T :\; \lef[{c}
                  \dot{q} \\
                  \dot{p}
                  \rig] = \lef[{cc}
                              0_{n\times n}              &  M^{-1} \, M_d \\
                           -M_d \, M^{-1}    &  \Lambda
                                                     \rig]  \nabla H_d ,
\endequ
and the stability condition \eqref{pmdmin} now becomes
\begequ
\lab{symlam}
p^\top M_d^{-1}(q) \Lambda(q,p) M_d^{-1}(q)p \leq 0.
\endequ
A sufficient, but not necessary, condition for \eqref{symlam} to hold is clearly
$$
\Lambda + \Lambda^\top \leq 0.
$$

Notice that, in contrast with the two step design procedure of standard IDA--PBC, in this new formulation the energy shaping and the damping injection are carried out {\em simultaneously}. This is in the spirit of \cite{BATetal} where it is shown that the partition into two steps of the design procedure induces some loss of generality.

On the other hand, it is easy to see (see \cite{CRAORTPIL}), that new KE--PDE becomes
\begequ
\sum_{i=1}^{n} \left[ (v_k^\top \, M_d  \, M^{-1} e_i) \, {\nabla_{q_{i}} M_d } - (v_k^\top \, e_{i}) \, M_{d} \, {\nabla_{q_{i}}M^{-1}} \, M_d \right] = - \sum_{i=1}^{n}  e_{i} \, v_k^\top \, Q_{i}(q),
\label{pde3}
\endequ
with $\;k \in \bar s$, and the control law takes the form
\begin{equation}
u =  (G^\top \, G)^{-1} \, G^\top \left[ \nabla_{q} H - M_{d} \, M^{-1} \, \nabla_{q}{H_d} + \Lambda \, M_d^{-1} \, p \right].
\label{usida}
\end{equation}
Similarly to classical IDA--PBC, the presence of the matrices $Q_i$ allows us to reduce the number of PDE's to be solved. Interestingly, this is equal to \eqref{numpde}, that is, the number of PDEs of IDA--PBC; see \cite{CHA}. In spite of this fact, we show in the next section---via a series of examples---that SIDA--PBC with generalised forces is applicable to a larger class of systems than standard IDA--PBC.

We wrap--up this section with a simple proposition that summarises the developments presented above and whose proof follows {\em verbatim} the proof of stability of standard IDA--PBC \cite{ORTtac}.
\begin{proposition}\em
\label{pro1}
Consider the underactuated mechanical system \eqref{sys} in closed--loop with the control \eqref{usida} verifying the following conditions.
\begenu
\item $H_d$  and $\Lambda$ are given by \eqref{hd} and \eqref{lam}, respectively.
\item $M_d$ and $\Lambda$ satisfy \eqref{symlam}.
\item $M_d,\;V_d$  and $Q$ verify the matching equations \eqref{pde2} and \eqref{pde3}.
\item$M_d$ is positive definite and $V_d$ satisfies \eqref{vdmin}.
\endenu
The closed--loop system takes the form \eqref{sigt} and it has a {\em globally stable} equilibrium at the desired point $(q,p)=(q_\star,0)$, with Lyapunov function $H_d$. The equilibrium is globally {\em asymptotically} stable if
$$
y_D:=(\Lambda+\Lambda^\top)^\hal  M_d^{-1}p
$$
is a detectable output of the closed--loop system.
\end{proposition}

%
\section{Examples of SIDA--PBC with Generalised Forces}
\label{sec4}
In this section we prove that several stabilising controllers for mechanical systems---that have been derived invoking other considerations---actually belong to the class of  SIDA--PBC with generalised forces presented in the previous section. More precisely, we prove that replacing the aforementioned state--feedback laws in the  system \eqref{sys} yields the desired target dynamics \eqref{sigt}, {\em i.e.}, that the matching equation \eqref{matequ1} holds.

The definition below is instrumental to articulate our results.

\begin{definition}\em
\lab{def1}
A state--feedback control law $u: \rea^n \times \rea^n \to \rea^m$  for the mechanical system \eqref{sys} is said to  be a  SIDA--PBC with generalised forces if the following identity holds true
\begequ
\lab{matequ}
- \nabla_q H(q,p) + G(q) u(q,p) = -M_d(q) M^{-1}(q) \nabla_q H_d(q,p) + \Lambda(q,p)M_d^{-1}(q) p
\endequ
where  $H_d$ is of the form \eqref{hd}, for some $M_d$ positive definite and $V_d$, $\Lambda$ verifying \eqref{vdmin} and  \eqref{symlam}, respectively. Such controllers ensure that the closed--loop system takes the pH form \eqref{sigt} and verify the conditions of Proposition \ref{pro1}.
\end{definition}
%
\subsection{Energy--shaping without solving PDEs: The controller of  \cite{Donaire2015}}
\lab{example1}
%
In \cite{Donaire2015}  a static state--feedback that assigns the Lyapunov function $H_d$  \eqref{hd}  for a class of mechanical systems was given. This control law does not satisfy the matching equation \eqref{matequ0}, therefore is not an IDA--PBC. However,  we show in this subsection that it does satisfy \eqref{matequ}---proving that it belongs to the class of  SIDA--PBC with generalised forces.

The design of \cite{Donaire2015} proceeds in two steps, first, a partial feedback linearization inner loop is applied to transform the system into Spong's Normal Form \cite{Spong1994}. Invoking Proposition 7 of \cite{Sarras2013}, conditions on $M$ and $V$ are imposed to ensure the partially linearized system is still a mechanical system. A consequence of the latter is the identification of two new cyclo--passive outputs based upon which the controller is designed in a second step. The derivations in \cite{Donaire2015} are done in the Lagrangian form, to fit it into the framework of this paper, we present below its pH formulation.

Consider a mechanical system \eqref{sys} with input matrix of the form
$$
G=\lef[{c} I_{m} \\ 0_{s \times m} \rig].
$$
Partition the generalised coordinates as $q=\col(q_a,q_u) $, with $q_a \in \RE^m$ and $q_u \in \RE^{s}$, which correspond to the actuated and unactuated coordinates, respectively. The inertia matrix is conformally partitioned as
 \begequarr
 M (q)= \left[ \begarr{cc} m_{aa} (q) & m_{au} (q) \\ m_{au}^\top (q) & m_{uu}  (q) \endarr \right],
 \endequarr
 where $m_{aa}: \rea^n \to \RE^{m\times m}$, $m_{au}:\rea^n \to  \RE^{m\times s}$ and $m_{uu}: \rea^n \to  \in \RE^{s\times s}$.

In Proposition 7 of \cite{Sarras2013} it is shown that the mechanical structure is preserved after partial feedback linearization if the following conditions are satisfied.
\begite
\item[{\bf A1.}] The inertia matrix depends only on the unactuated variables $q_u$, {\em i.e.},  $M(q)=M(q_u)$.
\item[{\bf A2.}] The sub--matrix $m_{aa}$ of the inertia matrix is constant.
\item[{\bf A3.}] The potential energy can be written as $V(q)=V_a(q_a)+V_u(q_u)$.
\item[{\bf A4.}] The rows of the matrix $m_{au}(q_u)$ satisfy
\begequ
\lab{maij}
\nabla_{q_{uj}} (m_{au})_k = \nabla_{q_{uk}} (m_{au})_j,\; \forall j\neq k,\; j,k \in \overline s.
\endequ
\endite
Under these conditions the system \eqref{sys} in closed--loop with the static state--feedback control law
\begequ
\lab{innoutloo}
u=u_{\tt PL}(q,p)+v,
\endequ
where $u_{\tt PL}:\rea^{n}\times \rea^n \to \rea^m$ is the partially linearizing feedback given in \cite{Spong1994}, see also Section VII of \cite{Donaire2015}, takes the pH form
\begin{eqnarray}
\label{newphsol}
\left[ \begarr{c} \dot{q} \\ \dot{\bfp} \endarr \right] = \left[ \begarr{cc} 0 & I_n \\ -I_n & 0 \endarr \right]  \nabla \tilde{H} + \left[ \begarr{c} 0 \\ \tilde{G}(q_u) \endarr \right]  v
\end{eqnarray}
$$
\tilde H(q,\bfp)=  \frac{1}{2} \bfp^\top \tilde M^{-1}(q_u) \bfp +V_u(q_u).
$$
where
$$
\tilde{M}(q_u)=\lef[{cc} I_m & 0 \\ 0 & m_{uu}(q_u)\rig],
$$
and
\begequ \label{gtilde}
\tilde{G}(q_u) := \left[ \begarr{c}  I_m \\ -m_{au}^\top(q_u)  \endarr \right].
\endequ
Notice that we have defined a new momenta via
$$
\bfp=\lef[{c} \bfp_a \\\bfp_u \rig]:=\tilde{M}(q_u)\dot{q}.
$$

To complete the controller design the following additional assumptions are made in  \cite{Donaire2015}.
\begite
\item[{\bf A5.}]  The columns of $m_{au}(q_u)$ are gradient vector fields, that is,
\begequ
\lab{colmau}
\nabla(m_{au})^i=[\nabla(m_{au})^i]^\top, \;\forall i \in \bar{m}.
\endequ
Equivalently, there exists a function $V_N:\RE^{s} \to \RE^m$ such that
\begequ
\lab{dotvn}
\dot{V}_N=-m_{au}(q_u)\dot{q}_u.
\endequ
\item[{\bf A6.}] There exist constants $k_e, k_a, k_u \in \RE$, $K_k,K_I \in \RE^{m\times m},\;K_k,K_I \geq 0$ such that the following holds.
\\ \ \\
{\bf (a)} $\det[K(q_u)]\neq 0,\;\forall q_u \in \RE^{s},$ where $K:\RE^{s} \to \RE^{m \times m}$ is defined as
\begequ
\lab{kqu}
K(q_u):= k_e I_m + k_a K_k + k_u K_k m_{au}(q_u) m_{uu}^{-1} (q_u)m_{au}^\top(q_u).
\endequ
{\bf (b)}  The matrix
\begequarr 
M_d^{-1}(q_u):= \left[ \begarr{cc} k_ek_a I_m + k_a^2 K_k &  \mathcal{X}(q_u) \\  \mathcal{X}^\top(q_u)  & \mathcal{Y}(q_u)   \endarr \right],  \label{mdpl}
\endequarr
with $\mathcal{X}(q_u)= -k_a k_uK_km_{au}(q_u)m^{-1}_{uu}(q_u)$ and $ \mathcal{Y}(q_u)=k_ek_um^{-1}_{uu}(q_u)+k_u^2m^{-1}_{uu}(q_u)m_{au}^\top(q_u) K_k m_{au}(q_u)m^{-1}_{uu}(q_u)$, is positive definite and the function
\begequarr \label{vdpl}
V_d(q):=k_ek_uV_u(q_u) + \frac 12 || k_aq_a+k_u V_N(q_u) ||^2_{K_I},
\endequarr
satisfies condition \eqref{vdmin}, and the minimum is isolated.
\endite

The following proposition is the main stabilization result of \cite{Donaire2015}.

\begin{proposition}\em
\label{pro2}
Consider the underactuated mechanical system \eqref{newphsol} with $m_{au}$, $m_{uu}$ and $V_u$ satisfying Assumptions {\bf A5} and {\bf A6}.
The control $v:\rea^{n}\times \rea^n \to \rea^m$ given by
\begequarr
\lab{vdon}
v(q,\bfp)\hspace{-3mm}&=& \hspace{-3mm} - K^{-1} \Bigg[ k_u K_k m_{au} m_{uu}^{-1} \nabla_{q_u} V_u + K_I (k_aq_a+k_uV_N) - \frac{k_u}{2} K_k m_{au}m_{uu}^{-1} \nonumber \\
&& \nabla^\top_{q_u} [m_{uu}^{-1} \bfp_u] \bfp_u + k_u K_k \nabla_q[m_{au}m_{uu}^{-1}\bfp_u] m_{uu}^{-1} \bfp_u \Bigg] - \nonumber \\
&& - K_P K^\top (k_a \bfp_a - k_u m_{au} m_{uu}^{-1} \bfp_u),
\endequarr
with $K_P > 0$ ensures that the closed--loop system has a {\em globally stable} equilibrium at the desired point $(q,\bfp)=(q_\star,0)$ with Lyapunov function
\begin{equation} \label{hdpfl}
H_d(q,\bfp)=\hal \bfp^\top M^{-1}_d\bfp+V_d(q),
\end{equation}
with $M_d^{-1}$ and $V_d$ given by \eqref{mdpl} and \eqref{vdpl}, respectively. The equilibrium is globally {\em asymptotically} stable if
$$
y_N:=k_a \bfp_a - k_u m_{au}(q_u) m_{uu}^{-1}(q_u) \bfp_u
$$
is a detectable output of the closed--loop system.
\end{proposition}

Now, we proceed to prove that the control \eqref{vdon} is a GSIDA--PBC with generalised forces. Towards this end, we first notice that the matching equation \eqref{matequ} for the system \eqref{newphsol} reduces to
\begequarr
-\frac 12 \nabla_q (\bfp^\top \tilde M^{-1} \bfp) -  \nabla_q V +  \tilde G v \hspace{-3mm}&= &\hspace{-3mm}- \frac 12 M_d \tilde M^{-1} \nabla_q(\bfp^\top M_d^{-1} \bfp) - M_d\tilde M^{-1} \nabla_q V_d + \nonumber \\
&&\hspace{-3mm}+ \Lambda M_d^{-1} \bfp 
\lab{newmatcon}
\endequarr
Hence,  we must prove that  \eqref{vdon} verifies \eqref{newmatcon} for some $\Lambda$ satisfying \eqref{symlam}. This fact is stated in the proposition below whose proof involves a series of long computations, therefore, it is given in \ref{appproofprop3}.
\begin{proposition}\em
\label{pro3}
Consider the underactuated mechanical system \eqref{newphsol} with $m_{au}$, $m_{uu}$ and $V_u$ satisfying Assumptions {\bf A5} and {\bf A6}.
The control \eqref{vdon} is a SIDA--PBC with generalised forces and
\begequarr
\Lambda(q,\bfp) &=& \frac 12 M_d \Bigg[  - M_d^{-1} \nabla_q^\top [\tilde{M}^{-1}\bfp] + \tilde{M}^{-1} \nabla_q^\top[M_d^{-1}\bfp] -  M_d^{-1} \tilde{G} K^{-1}
\nonumber \\
&& \hspace{-20mm}  \left[ \begarr{ccc }0 & \vdots & k_u K_k m_{au}m_{uu}^{-1} \nabla^\top_{q_u} [m_{uu}^{-1} \bfp_u]  - 2 k_u K_k \nabla_q[m_{au}m_{uu}^{-1} \bfp_u] m_{uu}^{-1}  \endarr \right] \Bigg] M_d - \nonumber  \\
&&  \hspace{-20mm} - \tilde{G}(q) K_P \tilde{G}^\top(q).
\lab{lambdapl} 
\endequarr
\end{proposition}
%
%
\subsubsection*{Application to the inverted pendulum on a cart}
\lab{subsubseccp}
To illustrate Proposition \ref{pro3} we consider here the controller for classical cart--pendulum example reported in \cite{Donaire2015}. This is a $2$--dof system with potential energy given by
$$
V(q_u)=mg\ell \cos(q_u),
$$
mass matrix 
$$
M(q_u) =\lef[{cc} M_c+m &  m\ell\cos(q_u) \\  m\ell\cos(q_u) &  m\ell^2  \rig],
$$
and the input matrix is $G=\col(1,0)$, where $q_a$ is the position of the car and $q_u$ denotes the angle of the pendulum with respect to the up-right vertical position. The parameter $M_c$ is the mass of the car,  $m$ is the mass of the pendulum and $\ell$ its length. The control objective is to stabilise the up-right vertical position of the pendulum. The system satisfies assumptions {\bf A1-A4}, thus, after using a partial-feedack linearising control \eqref{innoutloo}, the dynamics can be written as in \eqref{newphsol}, with momentum vector $\bfp=\col(\bfp_a,\bfp_u)=\col(\dot{q}_a,\frac{1}{m\ell^2}\dot{q}_u)$, $m_{uu}=m\ell^2$ and $m_{au}=m\ell\cos(q_u)$.

In \cite{Donaire2015} Proposition \ref{pro2} was used to derive the (locally stabilising) controller
\begequarr
\lab{concarpen}
v &=& \frac{1}{K(q_u)} \left[ -k_uK_km\sin(q_u)  \left(\frac{1}{m^2\ell^3} \bfp_u^2  - g \cos(q_u) \right)  \right]   -  K_p K(q_u) \nonumber \\
&& \left(\bfp_a- \frac{k_u}{\ell} \cos(q_u) \bfp_u \right)
\endequarr
where $K_I=0$, $k_a=1$ and 
\begequarr
\nonumber
K(q_u) & = & k_e  +  K_k + k_u K_k m \cos^2(q_u)\\
\nonumber
M_d^{-1} (q_u)& = & \left[ \begarr{cc} k_e +  K_k & - \frac{k_u K_k}{\ell} \cos(q_u) \\  -\frac{k_u K_k}{\ell} \cos(q_u) &  \frac{k_ek_u}{m\ell^2}+\frac{k_u^2 K_k}{\ell^2}\cos^2(q_u)   \endarr \right]\\
\lab{vdcarpen}
V_d(q_u) & = & k_e k_u mg\ell \cos(q_u).
\endequarr
The conditions of Proposition \ref{pro2} are satisfied if the controller gains verify 
$$
k_e > 0,\;k_u  <  0,\; k_e+K_k+ k_uK_k m \cos^2(q_u)  <  0,
$$
for all $q_u \in (-{\pi \over 2}, {\pi \over 2})$.  

Some lengthy, but straightforward, calculations show that the control law \eqref{concarpen} satisfies the matching condition \eqref{newmatcon} with $\Lambda$, derived from \eqref{lambdapl}, given by
\begequarr 
\label{lamcp}
 \Lambda(q,\bfp) &=& \frac 12 M_d  \left[ \begarr{cc} 0 & -\frac{2k_ak_uK_k}{ml^3} \sin(q_u) \bfp_u \\ \vspace{-5mm} \\ \\ \frac{k_ak_uK_k}{ml^3} \sin(q_u) \bfp_u & \frac{k_ak_uK_k}{ml^3} \sin(q_u) \bfp_a  \endarr \right]  M_d - \nonumber \\
 && - \left[ \begarr{c} 1 \\ -m \ell \cos(q_u) \endarr \right]  K_P \left[ \begarr{cc} 1 & -m \ell \cos(q_u) \endarr \right].
\endequarr
%
\subsection{Lyapunov approach for control of underactuated mechanical systems}
\lab{subsec42}
Several works have proposed an approach using direct Lyapunov method for control design of underactuated mechanicals system (see {\em e.g.} \cite{Aguilar2009,Turker2013,White2008}). In the following, we summarise the main idea proposed in these works.

Consider a mechanical system with dynamics as follows
\begequarr
\nonumber
\dot q &=& \calm^{-1}(q) \frp  \\
\dot \frp &=& g(q) + f(q,\frp) + \calG u.
 \label{mecp}
\endequarr
This dynamics could result from a change of coordinate or a preliminary feedback (or change of coordinates) on the mechanical system \eqref{sys} that may not preserve neither Lagrangian nor Hamiltonian structure. Notice that the system \eqref{mecp} coincides with the standard mechanical system \eqref{sys} if
\begequ
\lab{newdyn}
g(q) + f(q,\frp) \equiv -\nabla_q H(q,\frp),
\endequ
and $\calm$ is the inertia matrix.

To proceed with the design, the Lyapunov function candidate
\begequarr \label{lyap}
\calH_d(q,\frp)= \frac 12 \frp^\top \calm_d^{-1}(q) \frp + \calV_d(q),
\endequarr
with $\calm_d>0$ and $q_\star= \arg \min \calV_d(q)$ is proposed.
The control law is computed to ensure that the time derivative of the \eqref{lyap} along the dynamics \eqref{mecp} is negative semidefinite. That is,
\begequarr
\dot{\calH}_d&=& \frp^\top \calm_d^{-1} \left[ g(q) + f(q,\frp) +\calG u \right] + \frac 12 \nabla_q^\top [ \frp^\top \calm_d^{-1} \frp] \calm^{-1} \frp +  \nabla^\top \calV_d \calm^{-1} \frp \nonumber \\
&&\hspace{-15mm} = \frp^\top \calm_d^{-1} \left[ g(q) + f(q,\frp) + \calG u  + \frac 12 \calm_d \calm^{-1} \nabla_q^\top [ \calm_d^{-1} \frp]  \frp + \calm_d \calm^{-1} \nabla \calV_d  \right] 
\nonumber \\
&& \hspace{-15mm} \leq  0
\label{dlyap}
\endequarr
Now, define the vector
$$
C(q,\frp):= g(q) + f(q,\frp) + \calG u  + \frac 12 \calm_d \calm^{-1} \nabla_q^\top [ \calm_d^{-1} \frp]  \frp + \calm_d \calm^{-1} \nabla \calV_d  
$$
and, recalling observation {\bf O1}, rewrite it (without loss of generality) as 
$$
C(q,\frp) =  \Lambda(q,\frp) \calm_d^{-1} \frp.
$$
Replacing the equations above in \eqref{dlyap} yields
\begequ
\lab{conlamlya}
\frp^\top \calm_d^{-1} \Lambda(q,\frp) \calm_d^{-1} \frp \le 0,
\endequ
\begin{proposition}\em
The control law obtained via the so-called direct Lyapunov approach is a SIDA-PBC with generalized forces.
\end{proposition}
\begin{proof}
The proof follows noting that, from the derivations above, the control law should verify
\begequarr
g(q) + f(q,\frp) + \calG u   &=& - \frac 12 \calm_d \calm^{-1} \nabla_q^\top [ \calm_d^{-1} \frp]  \frp - \calm_d \calm^{-1} \nabla \calV_d + \Lambda \calm_d^{-1} \frp \nonumber \\
 &=& -  \calm_d \calm^{-1} \nabla_q \calH_d + \Lambda(q,\frp) \calm_d^{-1} \frp, \label{matequlyap}
\endequarr
which coincides with the matching equation \eqref{matequ}, if we consider a more general class of open-loop dynamics for the momenta. This matching equation together with the stability condition \eqref{conlamlya} shows that the controller is a SIDA-PBC with generalized forces, and the closed--loop takes the form \eqref{sigt}.
\qed
\end{proof}

\subsection*{Application to the ball and beam system}
We present here the 2-dof example of the ball and beam solved in \cite{Aguilar2009} using the direct Lyapunov method, and show that the resulting controller is a SIDA-PBC with generalised forces. The design in \cite{Aguilar2009} first applies a partial-feedback linearizing control and a change of coordinate that allows us to write the dynamics of the system as follows
\begequarr
\left[ \begarr{c} \dot{q}_a \\ \dot{q}_u \endarr \right] &=&  \left[ \begarr{cc}  \frac{1}{\sqrt{2(\epsilon+q_u^2) }} & 0 \\ 0 &1 \endarr \right]  \left[ \begarr{c} \frp_a \\ \frp_u \endarr \right]  \label{dqbb} \\
\left[ \begarr{c} \dot{\frp}_a \\ \dot{\frp}_u \endarr \right] &=&  \left[ \begarr{c}  0  \\ \frac{q_u \frp_a^2}{2(\epsilon +q_u^2)} - \delta \frp_u \endarr \right] +\left[ \begarr{c} 0 \\ -\sin(q_a)  \endarr \right] + \calG u, \label{dpbb}
\endequarr
where $q_a$ is the angle of the beam and $q_u$ is the position of the ball on the beam. The momentum vector is defined as $\frp=\calm(q_u) \dot{q}$, with $\calm(q_u)=\diag(\sqrt{2(\epsilon+q_u^2)},1)$, and $\calG=\col(1,0)$. The control objective is to stabilize the equilibrium $q_\star=(0,0)$. The Lyapunov function candidate has the form \eqref{lyap} with
$$
\calm_d^{-1}(q_u)= \left[ \begarr{cc}  \sqrt{2\epsilon+q_u^2} & -\sqrt{\epsilon +q_u^2} \\ -\sqrt{\epsilon +q_u^2} & \sqrt{2\epsilon+q_u^2}  \endarr \right],
$$
and
$$
\calV_d(q)= \epsilon \sqrt{2} [1-\cos(q_a)] + \frac{K}{2} \left[ q_a - \frac{1}{\sqrt{2}} \sinh^{-1} \left( \frac{q_u}{\sqrt{2\epsilon}} \right) \right].
$$
The controller proposed in \cite{Aguilar2009} is as follows
\begequarr 
u&= &- \frac{\sqrt{2\epsilon+q_u^2}}{\sqrt{\epsilon+q_u^2}} \sin(q_a) + \frac{1}{\sqrt{\epsilon+q_u^2}}\nabla_{q_u}\calV_d - c_a \frp_a - c_u \frp_u  - \nonumber \\
&&- \left( \delta + K_P \sqrt{2\epsilon+q_u^2} \right) \frp_a + K_P \sqrt{\epsilon+q_u^2} \frp_u , \label{conbb}
\endequarr
with the functions
\begequarrs
c_u(q,\frp) & := & -\frac{q_u \frp_u}{2 \sqrt{2\epsilon+q_u^2}\sqrt{\epsilon+q_u^2}}+\frac{q_u\frp_a}{2(\epsilon+q_a^2)}\\
c_a(q,\frp) &:= & - \frac{q_u\frp_a}{2\sqrt{2\epsilon+q_u^2}\sqrt{\epsilon+q_u^2}},
\endequarrs
and parameters $K$, $K_p$ and $\epsilon$ positive constants to be chosen.

We show in \ref{appbb} that the controller \eqref{conbb} satisfies the matching equation \eqref{matequlyap} with
\begequarr
\Lambda(q,\frp) \hspace{-3mm}&:=&\hspace{-3mm} - \half \calm_d \left[ \begarr{cc} 0 & -\frac{q_u \frp_u}{\sqrt{\epsilon+q_u^2}} + \frac{\sqrt{2\epsilon+q_u^2}q_u \frp_a}{(\epsilon+q_u^2)} \\ \frac{q_u \frp_u}{\sqrt{\epsilon+q_u^2}} - \frac{\sqrt{2\epsilon+q_u^2}q_u \frp_a}{(\epsilon+q_u^2)} & 0 \endarr \right] \calm_d- \nonumber \\
&&- \frac{1}{\epsilon} \left[ \begarr{cc}  \delta \sqrt{2\epsilon+q_u^2} + \epsilon K_P & \delta \sqrt{\epsilon+q_u^2} \\  \delta \sqrt{\epsilon+q_u^2} & \delta \sqrt{2\epsilon+q_u^2} \endarr \right].
\label{lambb}
\endequarr
We now verify that the matrix $\Lambda$ defined in \eqref{lambb} satisfies the stability condition \eqref{conlamlya}. For, we notice that the first matrix in  \eqref{lambb} is skew symmetric. Now, factoring the term  $\frac{\delta}{\epsilon}$, the second matrix can be partitioned as
$$
 \left[ \begarr{cc}  \sqrt{2\epsilon+q_u^2} + \epsilon K_P &  \sqrt{\epsilon+q_u^2} \\   \sqrt{\epsilon+q_u^2} &  \sqrt{2\epsilon+q_u^2} \endarr \right] = \left[ \begarr{cc}  \sqrt{2\epsilon+q_u^2}  &  \sqrt{\epsilon+q_u^2} \\   \sqrt{\epsilon+q_u^2} &  \sqrt{2\epsilon+q_u^2} \endarr \right] + \left[ \begarr{cc}  \epsilon K_P &  0 \\ 0   &  0\endarr \right] .
 $$
This matrix is positive definite because $\delta$, $K_P$ and $\epsilon$ are positive constants and the determinant of the first right hand matrix equals $\epsilon$. Therefore, the control law \eqref{conbb} is a SIDA-PBC, and the closed--loop dynamics can be written in the form \eqref{sigt}.
%
\section{Conclusions}
\lab{sec5}
%
An extension to the well known IDA--PBC method for mechanical systems has been reported. It essentially consists of two parts: (i) allowing the presence in the target dynamics of forces, which are more general than the usual gyroscopic ones, and (ii) the proposition of simultaneously carrying out the energy shaping and damping injection steps---instead of doing them as separate steps. These two modifications have been previously reported in \cite{CHA} and \cite{BATetal}, respectively. 

It has been shown that several recent controller designs that do not fit in the standard IDA--PBC paradigm, actually belong to this new  class of SIDA--PBC with generalised forces. In this way, it is shown that these controllers, that were derived invoking less systematic procedures,  are obtained following the well--established SIDA--PBC methodology.


%





%
\appendix
\section{Proof of Proposition \ref{pro3}}
\label{appproofprop3}

We state first the following lemma, whose proof is established using straightforward calculations, that will be used below.

\begin{lemma} \em
Given the matrices $M_d^{-1}$ and $\tilde G$ as in \eqref{mdpl} and \eqref{gtilde}, respectively. The following relation holds
\begequ \label{equlemma}
M_d^{-1} \tilde G = \left[ \begarr{c} k_a I_m \\ -k_u m_{uu}^{-1} m_{au}^\top  \endarr \right] K.
\endequ
\end{lemma}

The proof of proposition \ref{pro3} is divided in two parts. First, we verify that \eqref{vdon} satisfies the matching equation \eqref{matequ}. Second, we prove that $\Lambda$ given in \eqref{lambdapl} satisfies the stability condition \eqref{symlam}.

The matching equation \eqref{matequ} for the system \eqref{newphsol} is equivalent to
\begequarr
-\frac 12 M_d^{-1} \nabla_q [\bfp^\top \tilde M^{-1} \bfp ] - M_d^{-1} \nabla_q V + M_d^{-1} \tilde G v &= &- \frac 12 \tilde M^{-1} \nabla_q[\bfp^\top M_d^{-1} \bfp] - \nonumber \\
&-& \tilde M^{-1} \nabla_q V_d +M_d^{-1} \Lambda M_d^{-1} \bfp \nonumber
\endequarr
Then, the control law should satisy
\begequarr
M_d^{-1} \tilde G v \hspace{-2mm}&=&  \frac 12 M_d^{-1} \nabla_q^\top [ \tilde M^{-1} \bfp ] \bfp + M_d^{-1} \nabla_q V  - \frac 12 \tilde M^{-1} \nabla_q^\top[ M_d^{-1} \bfp] \bfp-  \nonumber \\
&& - \tilde M^{-1} \nabla_q V_d  +M_d^{-1} \Lambda M_d^{-1} \bfp \nonumber \\
&&\hspace{-20mm}=  \frac 12 M_d^{-1} \nabla_q^\top [ \tilde M^{-1} \bfp ] \bfp + M_d^{-1} \nabla_q V  - \frac 12 \tilde M^{-1} \nabla_q^\top[ M_d^{-1} \bfp] \bfp- \tilde M^{-1} \nabla_q V_d + \nonumber \\
&&\hspace{-20mm} - \frac 12 M_d^{-1} \tilde{G} K^{-1} \left[  k_u K_k m_{au}m_{uu}^{-1} \nabla^\top_{q_u} [m_{uu}^{-1} \bfp_u]  - 2 k_u K_k \nabla_q[m_{au}m_{uu}^{-1} \bfp_u] m_{uu}^{-1}  \right] \bfp_u   - \nonumber \\
&& \hspace{-20mm} - \frac 12 M_d^{-1} \nabla_q^\top [\tilde{M}^{-1}\bfp] \bfp + \frac 12 \tilde{M}^{-1} \nabla_q^\top[M_d^{-1}\bfp]  \bfp - M_d^{-1} \tilde{G} K_P \tilde{G}^\top M_d^{-1} \bfp  \nonumber \\
&&\hspace{-20mm}=   M_d^{-1} \nabla_q V  - \tilde M^{-1} \nabla_q V_d - M_d^{-1} \tilde{G} K_P \tilde{G}^\top M_d^{-1} \bfp - \nonumber \\
&& \hspace{-20mm} - \frac 12 M_d^{-1} \tilde{G} K^{-1} \left[  k_u K_k m_{au}m_{uu}^{-1} \nabla^\top_{q_u} [m_{uu}^{-1} \bfp_u]  - 2 k_u K_k \nabla_q[m_{au}m_{uu}^{-1} \bfp_u] m_{uu}^{-1}  \right] \bfp_u   \nonumber \\
 && \hspace{-20mm} = \left[ \begarr{c} -k_a k_u K_k m_{au} m_{uu}^{-1}  \\ k_ek_um_{uu}^{-1}+k_u^{2}m_{uu}^{-1}m_{au}^\top K_k m_{au} m_{uu}^{-1}  \endarr \right]  \nabla_{q_u}V - \left[ \begarr{c} 0_{m\times s}  \\ k_ek_um_{uu}^{-1}  \endarr \right]  \nabla_{q_u}V  - \nonumber \\
 &&  \hspace{-20mm} - \left[ \begarr{c} k_a I_m  \\ -k_u m_{uu}^{-1} m_{au}^\top \endarr \right]  K_I (k_aq_1+k_uV_N) - M_d^{-1} \tilde{G} K^{-1} \frac{k_u}{2} K_k m_{au}m_{uu}^{-1} \nonumber \\
 && \hspace{-20mm} \nabla^\top_{q_u} [m_{uu}^{-1} \bfp_u] \bfp_u + M_d^{-1} \tilde{G} K^{-1} k_u K_k \nabla_q[m_{au}m_{uu}^{-1} \bfp_u] m_{uu}^{-1}  \bfp_u -M_d^{-1} \tilde{G} K_P \tilde{G}^\top M_d^{-1} \bfp  \nonumber \\
&&\hspace{-20mm}= - \left[ \begarr{c} k_a I_m  \\ -k_u m_{uu}^{-1}m_{au}^\top  \endarr \right] k_u K_k m_{au} m_{uu}^{-1} \nabla_{q_u}V  - \left[ \begarr{c} k_a I_m  \\ -k_u m_{uu}^{-1} m_{au}^\top \endarr \right]  K_I   \nonumber \\
&& \hspace{-20mm} (k_aq_1+k_uV_N) - M_d^{-1} \tilde{G} K^{-1} \frac{k_u}{2} K_k m_{au}m_{uu}^{-1} \nabla^\top_{q_u} [m_{uu}^{-1} \bfp_u] \bfp_u + M_d^{-1} \tilde{G} K^{-1} k_u  \nonumber \\
&& \hspace{-20mm} K_k \nabla_q[m_{au}m_{uu}^{-1} \bfp_u] m_{uu}^{-1}  \bfp_u - M_d^{-1} \tilde{G} K_P \tilde{G}^\top M_d^{-1} \bfp  \nonumber \\
&& \hspace{-20mm} = - M_d^{-1} \tilde{G} K^{-1} \Bigg[ k_u K_k m_{au} m_{uu}^{-1} \nabla_{q_u}V  +  K_I (k_aq_1+k_uV_N) - \frac{k_u}{2} K_k m_{au}m_{uu}^{-1}  \nonumber \\
&& \hspace{-20mm} \nabla^\top_{q_u} [m_{uu}^{-1} \bfp_u] \bfp_u + k_u K_k \nabla_q[m_{au}m_{uu}^{-1} \bfp_u] m_{uu}^{-1}  \bfp_u \Bigg]  -M_d^{-1} \tilde{G} K_P K^\top \nonumber \\
&& \hspace{-20mm} (k_a \bfp_a - k_u m_{au} m_{uu}^{-1} \bfp_u) \nonumber 
\endequarr
\begequarr
M_d^{-1} \tilde G v \hspace{-3mm} &=&    M_d^{-1} \tilde{G} \Bigg\{- K^{-1} \Bigg[ k_u K_k m_{au} m_{uu}^{-1} \nabla_{q_u}V  +  K_I (k_aq_1+k_uV_N) - \nonumber \\
&&  - \frac{k_u}{2} K_k m_{au}m_{uu}^{-1} \nabla^\top_{q_u} [m_{uu}^{-1} \bfp_u] \bfp_u +  k_u K_k \nabla_q[m_{au}m_{uu}^{-1} \bfp_u] m_{uu}^{-1}  \bfp_u \Bigg]  \nonumber \\
&&  + K_P K^\top (k_a \bfp_a - k_u m_{au} m_{uu}^{-1} \bfp_u) \Bigg \}, \label{meveri}
\endequarr
where we used the definition of $\Lambda$ in the second equality, and the relation \eqref{equlemma} in the sixth equality. The control law \eqref{vdon} exactly coincides with the term in curly brackets hence it satisfies \eqref{meveri} and, therefore, the matching equation \eqref{matequ}.

Now, to prove that $\Lambda$ given in \eqref{lambdapl} satisfies \eqref {symlam}, we compute first some terms of $\Lambda$ as follows
\begequarr
\Delta_1 &=& -  M_d^{-1} \tilde{G} K^{-1} \left[ \begarr{ccc}0_{m\times m} & \vdots & \begarr{c} k_u K_k m_{au}m_{uu}^{-1} \nabla^\top_{q_u} [m_{uu}^{-1} \bfp_u]  - \\ - 2 k_u K_k \nabla_q[m_{au}m_{uu}^{-1} \bfp_u] m_{uu}^{-1} \endarr  \endarr \right]    \nonumber \\
\Delta_2 &=&- M_d^{-1} \nabla_q^\top [\tilde{M}^{-1}\bfp]  \nonumber  \\
\Delta_3 &=& \tilde{M}^{-1} \nabla_q^\top[M_d^{-1}\bfp]   \nonumber
\endequarr
from which we obtain
\begequarr
\Delta_1 \hspace{-3mm}&=&\hspace{-4mm}  \left[ \begarr{cc}  0_{m\times m}  & -k_ak_uK_k \left[ m_{au} m_{uu}^{-1} \nabla^\top_q(m_{uu}^{-1}\bfp_u)-2\nabla_q(m_{au}m_{uu}^{-1}\bfp_u)m_{uu}^{-1} \right] \\ \vspace{-2mm} \\  0_{s\times m} & k_u^2  m_{uu}^{-1} m_{au}^\top K_k  \left[ m_{au} m_{uu}^{-1} \nabla^\top_q(m_{uu}^{-1}\bfp_u)-2\nabla_q(m_{au}m_{uu}^{-1}\bfp_u)m_{uu}^{-1} \right]  \endarr \right]  \nonumber \\ \label{delta1} \\ \nonumber \\
\Delta_2 \hspace{-3mm}&=&\hspace{-4mm} \left[ \begarr{cc} 0_{m\times m}  & k_ak_uK_k m_{au} m_{uu}^{-1} \nabla^\top_q(m_{uu}^{-1}\bfp_u) \\ \vspace{-2mm} \\ 0_{s\times m} & -k_e k_u m_{uu}^{-1} \nabla^\top_q(m_{uu}^{-1}\bfp_u) - k_u^2 m_{uu}^{-1} m_{au}^\top K_km_{au} m_{uu}^{-1} \nabla^\top_q(m_{uu}^{-1}\bfp_u) \endarr \right] \nonumber \\ \label{delta2} \\ \nonumber \\
\Delta_3 \hspace{-3mm}&=&\hspace{-4mm} \left[ \begarr{cc} 0_{m\times m}  & 0_{m\times s} \\ \vspace{-2mm} \\ \hspace{-2mm} -k_ak_um_{uu}^{-1}\nabla^\top(K_km_{au}m_{uu}^{-1}\bfp_u) & \hspace{-2mm} \begarr{c} -k_ak_um_{uu}^{-1} \nabla^\top_q(m_{uu}^{-1}m_{au}^\top K_k \bfp_a) + \\  k_e k_u m_{uu}^{-1}  \nabla^\top_q(m_{uu}^{-1}\bfp_u) + \\  k_u^2  m_{uu}^{-1}  \nabla^\top_q(m_{uu}^{-1} m_{au}^\top K_k m_{au} m_{uu}^{-1} \bfp_u) \endarr  \endarr \right].   \nonumber \\  \label{delta3}
\endequarr
Now, we compute \eqref{symlam} using \eqref{lambdapl} and  \eqref{delta1}-\eqref{delta3}
\begin{eqnarray}
&&\hspace{-8mm}\eqref{lambdapl}= \bfp^\top [\Delta_1+\Delta_2+\Delta_3- M_d^{-1}\tilde{G} K_P \tilde{G}^\top M_d^{-1}] \bfp \nonumber \\
&& \hspace{-8mm} = \left[ \begarr{c} \bfp_a \\ \bfp_u \endarr \right]^\top  \left[ \begarr{c} 0_{m\times m} \\ \vspace{-2mm}   \\ -k_ak_um_{uu}^{-1}\nabla^\top(K_km_{au}m_{uu}^{-1}\bfp_u)  \endarr \right. \nonumber \\ \nonumber \\ 
&& \hspace{-12mm} \left. \begarr{c}  2 k_a k_u K_k \nabla^\top(m_{au}m_{uu}^{-1}\bfp_u) m_{uu}^{-1} \\ \vspace{-2mm} \\ \begarr{c} -2 k_u^2  m_{uu}^{-1} m_{au}^\top K_k \nabla^\top_q(m_{au} m_{uu}^{-1}\bfp_u) m_{uu}^{-1} - \\  -k_ak_um_{uu}^{-1} \nabla^\top_q(m_{uu}^{-1}m_{au}^\top K_k \bfp_a) +  \\   + k_u^2 m_{uu}^{-1}  \nabla^\top_q(m_{uu}^{-1} m_{au}^\top K_k m_{au} m_{uu}^{-1} \bfp_u) \endarr  \endarr \right] \left[ \begarr{c} \bfp_a \\ \bfp_u \endarr \right]  - \bfp^\top  M_d^{-1}\tilde{G} K_P \tilde{G}^\top M_d^{-1} \bfp \nonumber \\
&&\hspace{-8mm}= 2k_ak_u \left[ \bfp_a^\top K_k \nabla_q(m_{au}m_{uu}^{-1}\bfp_u) m_{uu}^{-1}\bfp_u - \bfp_a^\top \nabla_q (K_k m_{au}m_{uu}^{-1}\bfp_u) m_{uu}^{-1}\bfp_u  \right] -  \nonumber \\
&&\hspace{-8mm} -k_u^2\bfp_u^\top m_{uu}^{-1} \left[ 2\nabla_q^\top(m_{au} m_{uu}^{-1} \bfp_u) K_k m_{au} m_{uu}^{-1} \bfp_u-  \nabla^\top_q(m_{uu}^{-1} m_{au}^\top K_k m_{au} m_{uu}^{-1} \bfp_u)\right] - \nonumber \\
&&\hspace{-8mm} - \bfp^\top  M_d^{-1}\tilde{G} K_P \tilde{G}^\top M_d^{-1} \bfp \nonumber \\
&&\hspace{-8mm} = - \bfp^\top  M_d^{-1}\tilde{G} K_P \tilde{G}^\top M_d^{-1} \bfp, \nonumber
\end{eqnarray}
which shows that the condition \eqref{symlam} is satisfied since $K_P>0$.
%
\section{Matching Equation \eqref{matequlyap} for the Ball and Beam.} 
\label{appbb}
From the matching equation \eqref{matequlyap}, we obtain that the control law should satisfy the following equation
\begequarrs
\calG u  \hspace{-3mm}&=&\hspace{-3mm} - g(q) - \calm_d \calm^{-1} \nabla \calV_d - f(q,\frp) - \frac 12 \calm_d \calm^{-1} \nabla_q^\top [ \calm_d^{-1} \frp]  \frp  + \Lambda(q,\frp) \calm_d^{-1} \frp \nonumber \\
&&\hspace{-10mm}= \left[ \begarr{c} 0 \\ \sin(q_a)  \endarr \right] - \calm_d \calm^{-1} \nabla \calV_d -  \left[ \begarr{cc} 0 &  0  \\ \frac{q_u \frp_a}{2(\epsilon +q_u^2)} & 0 \endarr \right] \frp +  \left[ \begarr{cc}  0 & 0  \\ 0 & \delta \endarr \right] \frp - \nonumber \\
&&\hspace{-10mm} - \frac 12 \calm_d \calm^{-1} \nabla_q^\top [ \calm_d^{-1} \frp]  \frp + \Lambda(q,\frp) \calm_d^{-1} \frp \nonumber \\
&& \hspace{-10mm}= \left[ \begarr{c} 0 \\ \sin(q_a)  \endarr \right] -  \left[ \begarr{c} - \frac{1}{\sqrt{\epsilon+q_u^2}} \nabla_{q_u} \calV_d + \frac{\sqrt{2\epsilon+q_u^2}}{\sqrt{\epsilon+q_u^2}} \sin(q_a)\\ \sin(q_a)  \endarr \right]   -  \left[ \begarr{cc} 0 &  0  \\ \frac{q_u \frp_a}{2(\epsilon +q_u^2)} & 0 \endarr \right] \frp + \nonumber \\ 
&&\hspace{-10mm}+ \left[ \begarr{cc}  0 & 0  \\ 0 & \delta \endarr \right] \frp +  \frac{1}{2}  \left[ \begarr{cc}  \frac{q_u\frp_a}{\sqrt{\epsilon+q_u^2}\sqrt{2\epsilon+q_u^2}} & -\frac{q_u\frp_a}{(\epsilon+q_u^2)} + \frac{q_u \frp_u}{\sqrt{\epsilon+q_u^2}\sqrt{2\epsilon+q_u^2}}  \\ \frac{q_u \frp_a}{(\epsilon +q_u^2)} & 0 \endarr \right]  \frp - \nonumber \\
&&\hspace{-10mm}-  \left[ \begarr{cc}  \delta + K_P \sqrt{2\epsilon+q_u^2} & - K_P \sqrt{\epsilon+q_u^2} \\  0 & \delta  \endarr \right]  \frp \nonumber \\
&&\hspace{-10mm}=  \left[ \begarr{c}  \frac{1}{\sqrt{\epsilon+q_u^2}} \nabla_{q_u} \calV_d - \frac{\sqrt{2\epsilon+q_u^2}}{\sqrt{\epsilon+q_u^2}} \sin(q_a)\\ 0  \endarr \right]  - \left[ \begarr{cc}  \delta + K_P \sqrt{2\epsilon+q_u^2} & - K_P \sqrt{\epsilon+q_u^2} \\  0 & 0  \endarr \right]  \frp   + \nonumber \\
&& \hspace{-10mm} +   \left[ \begarr{cc}  \frac{q_u\frp_a}{2\sqrt{\epsilon+q_u^2}\sqrt{2\epsilon+q_u^2}} & -\frac{q_u\frp_a}{2(\epsilon+q_u^2)} + \frac{q_u \frp_u}{2\sqrt{\epsilon+q_u^2}\sqrt{2\epsilon+q_u^2}}  \\ 0 & 0 \endarr \right]  \frp \nonumber \\
\endequarrs
\begequarrs
&&\hspace{-40mm}=  \left[ \begarr{c}  \frac{1}{\sqrt{\epsilon+q_u^2}} \nabla_{q_u} \calV_d - \frac{\sqrt{2\epsilon+q_u^2}}{\sqrt{\epsilon+q_u^2}} \sin(q_a)\\ 0  \endarr \right]    -   \left[ \begarr{cc}  c_a & c_u  \\ 0 & 0 \endarr \right]  \frp - \nonumber \\
&& \hspace{-40mm} -  \left[ \begarr{cc}  \delta + K_P \sqrt{2\epsilon+q_u^2} & - K_P \sqrt{\epsilon+q_u^2} \\  0 & 0  \endarr \right]  \frp,
\endequarrs
which is satisfied by the control law \eqref{conbb}.


\end{document}